\UseRawInputEncoding 
\documentclass[oneside, 11pt]{amsart}
\usepackage{amsmath}
\usepackage{amssymb}
\usepackage{color}
\usepackage[usenames,dvipsnames,x11names,svgnames]{xcolor}
\usepackage[colorlinks=true,urlcolor=blue, linkcolor=NavyBlue,citecolor=DarkGreen]{hyperref}
\usepackage{amsthm}
\usepackage{amscd}
\usepackage{geometry}
\usepackage{mathrsfs}
\usepackage{dsfont}
 \usepackage{enumerate}

\newtheorem{thm}{Theorem}

\newtheorem{prop}{Proposition}
\newtheorem{lem}[prop]{Lemma}
\newtheorem{cor}[prop]{Corollary}

\theoremstyle{definition}

\newtheorem*{rem}{Remark}

\newcommand{\mb}{\mathbb}
\newcommand{\mc}{\mathcal}
\newcommand{\mf}{\mathfrak}
\newcommand{\ol}{\overline}

\newcommand{\leqs}{\leqslant }
\newcommand{\geqs}{\geqslant }
\newcommand{\wh}{\widehat}

\newcommand{\be}{\begin{equation*}}
\newcommand{\ee}{\end{equation*} }

\newcommand{\ben}{\begin{equation}}
\newcommand{\een}{\end{equation} }

\newcommand{\bs}{\begin{split}}
\newcommand{\es}{\end{split}}

\newcommand{\bmu}{\begin{multline*}}
\newcommand{\emu}{\end{multline*}}

\newcommand{\bmun}{\begin{multline}}
\newcommand{\emun}{\end{multline}}

\begin{document}
\title{Conditional mean values of long Dirichlet polynomials}
\author[W. Heap]{Winston Heap}
\address{Department of Mathematics, Shandong University, Jinan, Shandong 250100, China}
\email{winstonheap@gmail.com}
\maketitle 

\begin{abstract} Conditionally on the Riemann hypothesis we prove asymptotic formulae for mean values of various long Dirichlet polynomials involving the von Mangoldt function. Our results  avoid the use of correlation sum estimates although in addition to the Riemann hypothesis we must assume that our Dirichlet polynomials have weights from a specific class whose transforms are sufficiently concentrated near the origin. We also give large deviation estimates for these long Dirichlet polynomials. 
\end{abstract}

\section{Introduction}

Mean values of Dirichlet polynomials play a central role in analytic number theory. A fundamental formula, in a sharp form due to Montgomery--Vaughan \cite{MV}, states that for general coefficients $a(n)\in\mb{C}$,
\[
\int_T^{2T}\Big|\sum_{n\leqs X}a(n)n^{-it}\Big|^2dt=(T+O(X))\sum_{n\leqs X}|a(n)|^2.
\]
Thus, when $X=o(T)$ we acquire an asymptotic formula which can be viewed as an approximate form of Parseval's identity. However, when $X\gg T$ the resultant upper bound 
 may be far from the truth. Sums of length $X\gg T$ are typically referred to as long Dirichlet polynomials and the computation of their mean values feature throughout analytic number theory. In order to understand these it is typically required to have a good knowledge of the correlation sums 
\[
\sum_{n\leqs x}a(n)\overline{a(n+h)}
\] 
for a given $h$, which can be especially challenging if $a(n)$ has some multiplicative structure. 


Of particular interest to problems in prime number theory and the value distribution of the Riemann zeta function are the Dirichlet polynomials with coefficients given by the von Mangoldt function $\Lambda(n)$. Here, the above correlation sums are the subject of a specific case of the famous Hardy--Littlewood $k$-tuples conjecture stating that 
\begin{equation}\label{HL}
\sum_{n\leqs x}\Lambda(n)\Lambda(n+h)\sim \mf{G}(h)x
\end{equation}
for an explicitly given constant $\mf{G}(h)$. 
 Aside from being a deep result in it's own right, it can be shown \cite{Bolanz, GG} that in a strong form this would imply Montgomery's strong pair correlation conjecture (formula \eqref{mont conj} below) which would have further profound consequences for the distribution of primes and zeros of the Riemann zeta function \cite{GM,GG log der, Mont}.  
This seems currently out of reach, however there are some interesting results regarding \eqref{HL} on average \cite{E, MRT1, Mikawa}.

With strong error terms, formula \eqref{HL} and its generalisations to the $k$-fold convolution of the von Mangoldt function $\Lambda^{(k)}=\Lambda*\cdots *\Lambda$ would also allow the computation of the moments 
\begin{equation}\label{vM MV}
\int_T^{2T} \bigg|\sum_{n\leqs X}\frac{\Lambda(n)}{n^{1/2+it}\log n}\bigg|^{2k}dt
\end{equation}
beyond the ``short" range of $X \leqs  T^{1/k}$. From the work of Selberg on his central limit theorem \cite{Sel, Tsang thesis} we know that this polynomial provides a good approximation to $\log\zeta(\tfrac12+it)$ on average and plays a key role in its value distribution. This was further exemplified in the work of Soundararajan \cite{Sound} who gave Gaussian bounds for the large deviations of $\log|\zeta(\tfrac12+it)|$ and as a result proved the near sharp bounds 
\[
\int_T^{2T}|\zeta(\tfrac12+it)|^{2k}dt\ll T(\log T)^{k^2+\epsilon} 
\]
which was later refined by Harper \cite{Ha}. These last two papers, along with the work of Radziwi\l\l--Soundararajan \cite{RS1}, expanded on this idea and introduced new techniques. A great deal of applications ensued including  a new understanding of short interval maxima of the Riemann zeta function \cite{ABBRS, ABR, AOR}, value distribution of $L$-functions \cite{BELP, Das, HW,RSsel} and unconditional bounds for moments of zeta and $L$-functions \cite{Gao, HRS, HS} amongst others \cite{DFL,HR,LR}. Thus, through these techniques it seems a great deal can be extracted from just short Dirichlet polynomials although one barrier to further progress is naturally the issue of longer sums.

In this paper we are interested in this problem and, in particular, computing the mean values \eqref{vM MV} for $X$ beyond the range of $X\leqs T^{1/k}$. Our arguments do not make use of correlation sum estimates but instead assume the Riemann Hypothesis (RH) and crucially require the Dirichlet polynomial to be weighted. The weights can be taken from a reasonably wide class of functions, although they must have fairly specific properties and, in particular, have transforms that are sufficiently concentrated near the origin. The prototypical example we keep in mind is the pair
\begin{equation}\label{proto weight}
V_{X}(x)=\mathds{1}_{|x|\leqs \log X}\Big(1- \tfrac{|x|}{\log X}\Big),\qquad \widehat{V}_X(z)=\frac{\sin^2(\tfrac12 z\log X)}{(\tfrac12 z)^2\log X}
\end{equation} 
where here and throughout
\vspace{0.4cm}
\[
\widehat{V}_X(z)=\int_{-\infty}^\infty e^{ixz}V_X(x)dx
\]
\vspace{0.7cm}
denotes the Fourier transform. Before stating our results we describe this class of weights. 

We  consider functions with the following properties:

\begin{enumerate}[(i)]
\item \label{V cond} $V_X:\mathbb{R}\to\mathbb{R}$ is a bounded, even function with support in $[-\log X,\log X]$. 
\item \label{phi cond} There exists an $m\geqs 2$ such that for fixed $\Re(z)$ we have
\[
\widehat{V}_X(z/i)\ll \frac{X^{\Re(z)}}{|\Im(z)|^m}. 
\]
as $\Im(z)\to\infty$.
\item \label{f cond} For all $y\in\mb{R}$ we have 
\[
\frac{1}{\log X}\widehat{V}(\tfrac{y}{\log X})\ll \frac{1}{1+y^2}.
\]
\end{enumerate} 
The condition (\ref{phi cond}) allows for absolutely convergent contour integrals, and is perhaps not so crucial, whereas condition (\ref{f cond}) means that these integrals are sufficiently concentrated near the origin and seems harder to dispense with. For long Dirichlet polynomials with weights from this class we prove the following.

\begin{thm}\label{moments thm}Assume RH and suppose $V_X:\mathbb{R}\to\mathbb{R}$ satisfies conditions (\ref{V cond})--(\ref{f cond}). Then for $X\leqs T^{2m}$, $k\in\mb{N}$ and fixed  $0<\theta<1/100$  we have 
\begin{equation*}
\int_T^{2T} \bigg|\sum_{n\leqs X}\frac{\Lambda(n)V_X(\log n)}{n^{1/2+it}\log n}\bigg|^{2k}dt
=
k!\,T\bigg(\sum_{p\leqs T^{\theta/k}}\frac{V_X(\log p)^2}{p}\bigg)^k+O(T(Ck)^{4k}(\log\log T)^{k-1/2}) 
\end{equation*}
for some positive constant $C$.
For $k\in\tfrac12\mb{N}$ we have
\begin{equation*}
\int_T^{2T} \bigg(\mathfrak{F}\sum_{n\leqs X}\frac{\Lambda(n)V_X(\log n)}{n^{1/2+it}\log n}\bigg)^{2k}dt
=
c_k\,T\bigg(\sum_{p\leqs T^{\theta/k}}\frac{V_X(\log p)^2}{p}\bigg)^k+O(T(C^\prime k)^{4k}(\log\log T)^{k-1/2}) 
\end{equation*}
for some positive constant $C^\prime$ where $\mathfrak{F}$ denotes either the real or imaginary part and 
\[
c_k=
\begin{cases}
\frac{(2k)!}{2^{2k}k!} & 2k \text{ is even, }\\
0 & 2k \text{ is odd. }
\end{cases}
\]
\end{thm}

Here, we see that the Gaussian behaviour persists for quite long sums and thus the well known heuristic that the $(p^{-it})_{p \text{ prime}}$ act as independent random variables seems fairly robust in this context. 
We also note that the main terms are in the form of purely diagonal contributions. Thus, if one computes these mean values in the usual way using correlation sum estimates such as \eqref{HL}, the above result shows that on RH the off-diagonal terms cannot contribute to the leading order in this case. Furthermore, it suggests that the error terms in formula \eqref{HL} and its generalisations to $\Lambda^{(k)}$ must exhibit some additional cancellation when averaged over $h$ since their best-possible pointwise bound is likely squareroot cancellation and this would not be sufficient for sums of length $X^k\geqs T^2$.

Our methods extend to other Dirichlet polynomials and, in particular, a weighted version of the Dirichlet polynomial associated to the logarithmic derivative $(\zeta^\prime/\zeta)(s)$:
\[
\sum_{n\leqs X}\frac{\Lambda(n)}{n^{1/2+it}}.
\]
In this case we are restricted to the mean square for reasons we shall explain below.  Contrary to the above case, our main term contains an off-diagonal type contribution involving a well-known function from Montgomery's work on the pair correlation of zeta zeros \cite{Mont}, namely, 
\begin{equation}\label{F}
F(u)
=
F(u,T)
=
\frac{2\pi}{T\log T}\sum_{0<\gamma_1,\gamma_2\leqs T}T^{-iu(\gamma_1-\gamma_2)}w(\gamma_1-\gamma_2)
\end{equation}
where the sum is over the ordinates $\gamma$ of the non-trivial zeros $\rho=1/2+i\gamma$ of  $\zeta(s)$ and $w(u)=4/(4+u^2)$. Our result, which we state for the specific weight in \eqref{proto weight}, is then as follows. 

\begin{thm}\label{log der thm}  Assume RH and suppose $X\leqs T^4$. Then
\begin{multline*}
\int_T^{2T}\bigg|\sum_{n\leqs X}\frac{\Lambda(n)}{n^{1/2+it}}\Big(1-\tfrac{\log n}{\log X}\Big)\bigg|^2dt
=
T\sum_{p\leqs \min(T,\,X)}\frac{\log^2 p}{p}\Big(1-\tfrac{\log p}{\log X}\Big)^2
\\
+
\mathds{1}_{X\geqs t}\cdot{t\log^2 t}\int_1^{\tfrac{\log X}{\log t}}F(u,t)\big(1-\tfrac{u\log t}{\log X}\big)^2du\,\bigg|_{t=T}^{2T}+o(T(\log T)^2)
\end{multline*}
where $\mathds{1}$ denotes the characteristic function.
\end{thm}


 Regarding the function $F(u)$, Montgomery \cite{Mont} showed that 
\begin{equation}\label{F 1}
F(u)
=
u+o(1)+(1+o(1))T^{-2u}\log T
\end{equation}
uniformly for $0\leqs u\leqs 1$ as $T\to\infty$ (see the comment after formula (4.6) in \cite{Goldston}). We also have $F(u)\geqs 0$ and $F(-u)=F(u)$ (see \cite{HB}). Montgomery conjectured that for $u\geqs 1$, 
\begin{equation}\label{mont conj}
F(u)=1+o(1)
\end{equation}
as $T\to\infty$, and this implies the famous pair correlation conjecture for the zeros of $\zeta(s)$.  Note that for us, it would imply an  asymptotic of the form $\sim cT(\log T)^2$ 
in Theorem \ref{log der thm}. 
However, Montgomery derived the conjecture in \eqref{mont conj} by assuming \eqref{HL} in a strong form and this would likely allow for a direct proof of an asymptotic in Theorem \ref{log der thm} anyway. 
In lieu of  \eqref{mont conj} we shall have use for the following type of result which states that on RH,   
\begin{equation}\label{F int}
c_1<\int_{b}^{b+1}F(u)du<c_2
\end{equation} 
uniformly in $b$ (possibly dependent on $T$) for some constants $c_1,c_2>0$. See Lemma A of \cite{Goldston} or the Lemma of \cite{GG log der} for example. From this we can at least acquire the order.  

\begin{cor}Assume RH. Then for $T\leqs X\leqs T^4$ we have
\[
\int_T^{2T}\bigg|\sum_{n\leqs X}\frac{\Lambda(n)}{n^{1/2+it}}\Big(1-\tfrac{\log n}{\log X}\Big)\bigg|^2dt
\asymp T(\log T)^2.
\]
\end{cor}

Returning to Theorem \ref{moments thm}, we note the dependency on $k$ in the error terms is rather weak, and in fact, is too weak to get good large deviation bounds via Chebyshev's inequality. However, using similar tools and ideas from the proof of Theorem \ref{moments thm} we can find more efficient methods to compute the large deviations of these Dirichlet polynomials. Owing to large negative values of $\log|\zeta(\tfrac12+it)|$ we can only prove Gaussian bounds for large \emph{positive} values of the real part of our sum. The imaginary part is not limited in this way. 

\begin{thm}\label{large dev thm}Assume RH. Suppose that $\wh{V}_X(y)$ satisfies properties (\ref{V cond})--(\ref{f cond}) and that $\wh{V}_X(y)$ is positive. 
Then for $X\leqs T^{2m}$, $\epsilon>0$ and $W\geqs \sqrt{\log\log T}$ 
\begin{multline*}
\frac{1}{T}\mu\bigg(t\in [T,2T]: \bigg|\Im\sum_{n\leqs X}\frac{\Lambda(n)V_X(\log n)}{n^{1/2+it}\log n}\bigg|\geqs W \cdot V(0)\bigg)
\\
\ll \exp(-(1-\epsilon)W^2/\log\log T) + \exp(-C_\epsilon W\log W)
\end{multline*}
and 
\begin{multline*}
\frac{1}{T}\mu\bigg(t\in [T,2T]: \Re\sum_{n\leqs X}\frac{\Lambda(n)V_X(\log n)}{n^{1/2+it}\log n}\geqs W \cdot V(0)\bigg)
\\
\ll \exp(-(1-\epsilon)W^2/\log\log T) + \exp(-D_\epsilon W\log W)
\end{multline*}
for some constants $C_\epsilon, D_\epsilon>0$ where $\mu$ denotes Lebesgue measure.
\end{thm}

\begin{rem}
For the imaginary part it is not strictly necessary to have $\wh{V}_X(y)$ positive. If exception to this occurs then the same bound holds provided we replace $V(0)$ on the left hand side with $\frac{1}{2\pi}\int_{-1}^1|\wh{V}_X(y)|dy$.    
\end{rem}

We note that these results show some similarity with Soundararajan's \cite{Sound} large deviation bounds for $\log|\zeta(\tfrac12+it)|$ and this is no coincidence since we make key use of his arguments. Note that the Gaussian term dominates for $W\ll\log_2 T\log_3T$. Previously, results of this quality in such a range would only be accessible for $X\leqs T^{1/(\log_2 T)(\log_3 T)^2}$. Also, the bound  $\ll e^{-CW\log W}$ holds in the full range of $W$ with such results previously requiring $X\ll (\log T)^\theta$ for some $\theta\leqs 2$ (see Lemma 3 of \cite{H split} for example). 
It is possible that one could use these bounds  to help compute the characteristic/moment generating functions of these polynomials for reasonably large $X$.  

Turning to the proofs, we first note the following observation which provides a loose basis for our arguments. For large $t\in [T,2T]$ the explicit formula roughly states that
\[
\log\zeta(\tfrac12+it)\approx \sum_{p\leqs X}\frac{1}{p^{1/2+it}} + \sum_{\rho}\int_{1/2}^\infty \frac{X^{\rho-\sigma-it}}{\rho-\sigma-it}d\sigma
\]
for a given parameter $X\geqs 2$. 
Although the sum over zeros can get large pointwise, on average it is of the order $\log T/\log X$ since this is the expected number of zeros in the window $|\gamma-t|<1/\log X$ and the exponential integral localises the sum to such $\gamma$. Morally, this fact is independent of whether $X\leqs T$ or $X>T$.
Thus, on comparing the explicit formula with two different parameters $X,Y$ with $Y=T^\epsilon$, say, we can expect that on average 
\begin{equation}\label{heuristic}
\sum_{p\leqs X}\frac{1}{p^{1/2+it}}\approx \sum_{p\leqs Y}\frac{1}{p^{1/2+it}}+O(1),
\end{equation}    
that is, our long Dirichlet polynomial can be replaced by a short one in the mean.  

In practice our proofs run slightly differently. We avoid dealing directly with zeros by expressing our sum as a contour integral and shifting close to the half-line, but not past it. Truncating the integral at a low height and raising it to the $2k$th power we see that we need to compute the shifted  $2k$th moment of $\log\zeta(\tfrac12+it+z_j)$ for small shifts $z_j$. Here, we can express things in terms of short sums by approximating the logarithm in the mean with
\[
\sum_{p\leqs T^{1/k}}\frac{1}{p^{1/2+it+z_j}}+O(1).
\] 
Then once the shifted $2k$th moment has been computed it is a fairly simple matter to perform the $z_j$ integrals and we acquire our asymptotic formula. The properties of our weights allow us to carry out this procedure without any real losses, which would not be the case for the usual Perron's formula.

For the Dirichlet polynomial associated to the logarithmic derivative we can integrate by parts to get a similar contour integral involving $\log\zeta(\tfrac12+it+z_j)$, although we gain a factor of $\log X$ in the process. This is the expected average order of magnitude which means we must be more precise in our shifted moment formulae. In particular,  we are required to give asymptotics for the shifted second moment \[\int_T^{2T}\log\zeta(\tfrac12+it+z_1)\log\zeta(\tfrac12+it+z_2)dt\] down to the level $o(T)$. By using a further contour integration argument we can reduce this to the study of 
\[
\int_T^{2T}S(t+y_1)S(t+y_2)dt
\]
where $S(t)=\tfrac1\pi\Im\log\zeta(\tfrac12+it)$ and $|y_j|\ll \log T$. For this we apply the methods of Goldston \cite{Goldston} which give lower order terms in the second moment of $S(t)$. Our precise result is stated in Proposition \ref{S corr prop} below. The process of specialising to $S(t)$ entails some extra mild conditions on the weights and so in an effort to keep the exposition simple we have restricted ourselves to \eqref{proto weight}, although Theorem \ref{log der thm} should hold for more general weights.


 It would of course be interesting to remove the assumptions of the Riemann hypothesis and the weights in our results, although it seems difficult to give good mean bounds for 
\[
\log\zeta(\tfrac12+it)-\sum_{p\leqs X}\frac{V_X(\log p)}{p^{1/2+it}}
\]  
when $X\geqs T$ without  such assumptions. Also, a relaxing of condition (\ref{f cond}) would be welcome since, as it stands, common weights such as smooth approximations to the characteristic function of an interval are excluded.

\vspace{0.5cm}
 {\bf Acknowledgments.} The author thanks Pavel Mozolyako for the invitation to the Euler Mathematics Institute, St. Petersburg, where this work was carried out under the support of the Ministry of Science and Higher Education of the Russian Federation, agreement no. 075-15-2019-1619.




\section{Convolution formulae}\label{Conv sec}

\begin{lem}\label{conv lem}
Assume RH and suppose $V_X$ satisfies property (\ref{phi cond}). Then for $1/2\leqs \sigma\leqs 1$ and $t\geqs 1$ we have 
\begin{equation}\label{R}
\sum_{n\geqs 2}\frac{\Lambda(n)V_X(\log n)}{n^{\sigma+it}\log n}
=
\frac{1}{2\pi}\int_{-\infty}^{\infty}\log\zeta(\sigma+i(y+t))\widehat{V}_X(y)dy
+
O\big(\tfrac{X^{1-\sigma}}{t^m}\big).
\end{equation}
\end{lem}
\begin{proof}
This is essentially Lemma 5 of Tsang \cite{Tsang}. We shall give details for completeness. First note that by Fourier inversion
\[
V_X(\log n)=\frac{1}{2\pi }\int_{-\infty}^\infty \widehat{V}_X(x)n^{-ix}dx
=
\frac{1}{2\pi i}\int_{(c)}\widehat{V}_X(\tfrac{z}{i})n^{-z}dz 
\] 
for any $c>0$ after shifting contours. Here and throughout $\int_{(c)}$ denotes the integral $\int_{c-i\infty}^{c+i\infty}$. Choosing $c=1-\sigma+1/\log X$ and interchanging the order of sum and integral we get 
\[
\sum_{n\geqs 2}\frac{\Lambda(n)V_X(\log n)}{n^{\sigma+it}\log n}
=
\frac{1}{2\pi  i}\int_{(c)} \log\zeta(\sigma+it+z)\widehat{V}_X(\tfrac zi)dz.
\]
Truncating the integral at height $\Im(z)=\pm Z$ incurs an error of size $\ll X^{1-\sigma}\log_2 X/Z^{m-1}$ by property (\ref{phi cond}) and the bound $\zeta(1+1/\log X+iy)\ll \log X$ $\forall y\in\mb{R}$. We then consider the resultant integral as part of the rectangular contour with vertices at $\pm iZ$, $1-\sigma+1/\log X\pm iZ$ whose left edge has small semicircular indentations excluding the zeros at $z=i(\gamma-t)$  if $\sigma=1/2$ and which has a line from $1-\sigma-it$ to $1-it$, a small circular contour enclosing the singularity at $z=1-it$, and then a line back to the left edge of the contour.

The upper and lowermost horizontal components of the contour contribute $\ll X^{1-\sigma}Z^{-m}$ $ \log(Z+t)$ since $\int_\sigma^2 |\log\zeta(\alpha+iy)|d\alpha \ll \log y$ for all large $y$ and $\sigma\geqs 1/2$ (see formula \eqref{9.6} below for example). Thus we can let $Z\to\infty$. 

By property (\ref{phi cond}), the Hankel contour enclosing $z=1-it$ contributes $\ll X^{1-\sigma}t^{-m}$. On the remaining left edge of the contour, if $\sigma=1/2$, we can let the radii of the semicircular indentations tend to zero (indeed, the singularities are logarithmic and hence integrable). The result then follows. 
\end{proof}

Our next lemma truncates the integral in \eqref{R} to a suitable height when $X$ is sufficiently large in terms of $t$.

\begin{lem}\label{trunc lem}
Assume RH and suppose $V$ satisfies (\ref{f cond}). Then for $1/2\leqs \sigma\leqs 1$, large $t\in [T,2T]$ and $X$ satisfying $\log X\asymp \log T$ we have 
\begin{equation}\label{trunc}
\int_{-\infty}^{\infty}\log\zeta(\sigma+i(y+t))\widehat{V}_X(y)dy
=
\int_{-1}^{1}\log\zeta(\sigma+i(y+t))\widehat{V}_X(y)dy
+O(1)
\end{equation}
\end{lem}
\begin{proof}

By property (\ref{f cond}) of $\widehat{V}_X$ we have
\[
\int_1^\infty \log\zeta(\sigma+i(y+t))\widehat{V}_X(y)dy
\ll 
\frac{1}{\log X}\sum_{n=1}^\infty \frac{1}{n^2}\int_n^{n+1}|\log\zeta(\sigma+i(y+t))|dy.
\]
Formula 9.6 (B) of \cite{T} states that
\begin{equation}\label{9.6}
\log\zeta(\sigma+iu) 
=
\sum_{|\gamma-u|\leqs 1}\log(\sigma-\tfrac12+i(u-\gamma))+O(\log(2+|u|))  
\end{equation}
for $-1\leqs \sigma\leqs2$ and large $u$. The above integral is thus
\begin{align*}
\ll &
\int_n^{n+1}\sum_{|\gamma-(y+t)|\leqs 1}|\log(\sigma-\tfrac12+i(y+t-\gamma))|dy+\log(t+n)
\\
\ll &
\sum_{|\gamma-(n+t)|\leqs 2}\int_{-1}^1 |\log(\sigma-\tfrac12+iy)|dy+\log(t+n)\ll \log(t+n).
\end{align*}
Therefore, our tail integral is $\ll 1$ and of course the same bound holds for the other tail integral, thus giving the result.
\end{proof}

We combine Lemmas \ref{conv lem} and \ref{trunc lem} into the following key proposition. 

\begin{prop}\label{key prop} Assume RH and suppose $V_X$ satisfies (\ref{phi cond}) and (\ref{f cond}). Then for large $t\in[ T,2T]$ and $T^\epsilon\leqs X\leqs T^{2m}$,
\begin{equation}\label{sum form}
\sum_{n\geqs 2}\frac{\Lambda(n)V_X(\log n)}{n^{\sigma+it}\log n}
=
\frac{1}{2\pi}\int_{-1}^1\log\zeta(\sigma+i(y+t))\widehat{V}_X(y)dy
+
O(1).
\end{equation}
\end{prop}

For the Dirichlet polynomial associated with the logarithmic derivative we have the following result. 

\begin{prop}\label{log der prop}Assume RH. Let $t\in[T,2T]$ and $T^\epsilon\leqs X\leqs T^4$. Then 
\[
\sum_{n\leqs X}\frac{\Lambda(n)}{n^{1/2+it}}(1-\tfrac{\log n}{\log X}) 
=
\int_{-\log T/\log_2 T}^{\log T/\log_2 T}S(t+y)f_X(y)dy+O(1)
\]
where 
\[
f_X(y)=\frac{X^{iy}(2-iy\log X)-2-iy\log X}{y^3\log X}
\]
\end{prop}

\begin{proof}
By Mellin inversion we have 
\[
\sum_{n\leqs X}\frac{\Lambda(n)}{n^{1/2+it}}(1-\tfrac{\log n}{\log X}) 
=
\frac{1}{2\pi i \log X}\int_{(c)}-\frac{\zeta^\prime}{\zeta}(s+\tfrac12+it)  \frac{X^s}{s^2}ds
\]
for any $c>1/2$. For brevity let $\log X=\mc{L}$. Shifting to the line with real part $1/\mc{L}$ we encounter a pole at $s=1/2-it$ and the above is 
\begin{equation}\label{log der 1}
\frac{1}{2\pi i \mc{L}}\int_{(1/\mc{L})}-\frac{\zeta^\prime}{\zeta}(s+\tfrac12+it)  \frac{X^s}{s^2}ds+O\Big(\frac{X^{1/2}}{t^2\mc{L}}\Big). 
\end{equation}
By Theorem 9.6(A) of \cite{T} we have 
\[
-\frac{\zeta^\prime}{\zeta}(z)=\sum_{|\gamma-\Im(z)|\leqs 1}\frac{1}{z-\tfrac12-i\gamma}+O(\log(2+|\Im(z)|)) 
\]
and consequently $(\zeta^\prime/\zeta)(s+\tfrac12+it)\ll \mc{L}\cdot\log(|s|+t)$ for $\Re(s)=1/\mc{L}$. Thus, the expression in \eqref{log der 1} is 
\begin{align*}
&
\frac{1}{2\pi i \mc{L}}\int_{1/\mc{L}-iT/2}^{1/\mc{L}+iT/2}-\frac{\zeta^\prime}{\zeta}(s+\tfrac12+it)  \frac{X^s}{s^2}ds
+O\Big(\frac{(\log T)^2}{T}\Big)
+O\Big(\frac{X^{1/2}}{t^2\mc{L}}\Big).
\end{align*}

From Theorem 14.14 (B) of \cite{T} we have 
\begin{equation}\label{log zeta}
\log\zeta(\sigma+i\tau)\ll \frac{\log \tau}{\log_2\tau}\log\Big(\frac{2}{(\sigma-1/2)\log_2\tau}\Big),\qquad \tfrac{1}{2}<\sigma\leqs \tfrac12+C/\log_2\tau
\end{equation}
for large $\tau$.
Thus on integrating by parts the above integral is given by 
\begin{equation}\label{log der 2}
\frac{1}{2\pi i \mc{L}}\int_{1/\mc{L}-iT/2}^{1/\mc{L}+iT/2}\log{\zeta}(s+\tfrac12+it) \frac{d}{ds}\bigg[\frac{X^s}{s^2}\bigg]ds
+
O\Big(\frac{(\log T)^2}{T}\Big).
\end{equation}

Now, formula (14.10.5) of \cite{T} with $\alpha\to1/2^+$ gives 
\[
\log\zeta(s+\tfrac{1}{2}+it)
=
i\int_{Y}^{Y^\prime}\frac{S(y)}{s+i(t-y)}dy+O\Big(\frac{\log Y}{\Im(s)+t-Y}\Big)+O\Big(\frac{\log Y^\prime}{Y^\prime-(t+\Im(s))}\Big).
\]
for $Y<\Im(s)+t<Y^\prime$. Choosing $Y=T/3$ and $Y^\prime=3T$ the above error terms are $\ll \log T/T$ and so \eqref{log der 2} is 
\[
i\int_{T/3}^{3T}S(y) I(t-y)dy
+
O\Big(\frac{(\log T)^2}{T}\Big).
\]
where 
\[
I(t-y)=\frac{1}{2\pi i \mc{L}}\int_{1/\mc{L}-iT/2}^{1/\mc{L}+iT/2}\frac{1}{s+i(t-y)}\frac{d}{ds}\bigg[\frac{X^s}{s^2}\bigg]ds.
\]
Extending the integrals to $\pm i\infty$ incurs an error $\ll T^{-2}\log T$. Then, shifting the contour to the left we pick up poles at $s=-i(t-y)$ and $s=0$ to find 
\begin{align*}
I(t-y)
= &
\frac{1}{\mc{L}}(\mathrm{Res}_{s=-i(t-y)}+\mathrm{Res}_{s=0})\bigg[\frac{X^s(s\mc{L}-2)}{s^3(s+i(t-y))}\bigg]+O\Big(\frac{\log T}{T^2}\Big)
\\
= &
\frac{1}{\mc{L}}\bigg(\frac{X^{-i(t-y)}(-i(t-y)\mc{L}-2)}{(-i(t-y))^3}+\frac{1}{2!}\frac{d^2}{ds^2}\frac{X^s(s\mc{L}-2)}{s+i(t-y)}\bigg|_{s=0}\bigg)+O\Big(\frac{\log T}{T^2}\Big)
\\
= &
\frac{1}{\mc{L}}\bigg(\frac{X^{-i(t-y)}(-i(t-y)\mc{L}-2)}{(-i(t-y))^3}+\frac{\mc{L}}{(i(t-y))^2}-\frac{2}{(i(t-y))^3}\bigg)+O\Big(\frac{\log T}{T^2}\Big)
\\
= &
\frac{i(t-y)\mc{L}-2+X^{-i(t-y)}(i(t-y)\mc{L}+2)}{\mc{L}(i(t-y))^3}+O\Big(\frac{\log T}{T^2}\Big).
\end{align*}
Substituting $y\mapsto y+t$ and noting that $iI(-y)$ is given by $f_X(y)$ plus a negligible error term we find 
\[
\sum_{n\leqs X}\frac{\Lambda(n)}{n^{1/2+it}}(1-\tfrac{\log n}{\log X}) 
=
\int_{T/3-t}^{3T-t}S(t+y)f_X(y)dy+O\Big(\frac{(\log T)^2}{T}\Big).
\]
Truncating the integral at $\pm \log T/\log_2 T$ incurs an error $\ll 1$ since $S(t+y)\ll \log T/\log_2 T$ in the range of integration and $f_X(y)\ll y^{-2}$ for $y\gg 1$. The result then follows.
\end{proof}

\section{$2k$th Moment formulae: Proof of Theorem \ref{moments thm}} 

\subsection{The $2k$th moment of the real and imaginary parts} We consider the real part since the case of the imaginary part is analogous. Suppose $t\in [T,2T]$ and $T^{\epsilon}\leqs X\leqs T^{2m}$. Since $V_X$ is an even function, $\widehat{V}_X(y)$ is real and hence  
\begin{equation}\label{sum form real}
\Re\sum_{n\geqs 2}\frac{\Lambda(n)V_X(\log n)}{n^{1/2+it}\log n}
=
\frac{1}{2\pi}\int_{-1}^1\log|\zeta(\tfrac12+i(y+t))|\widehat{V}_X(y)dy
+
O(1)
\end{equation}
by Proposition \ref{key prop}.
 Then for even $2k$, by H\"older's inequality we have 
\begin{equation}\label{2kth}
\int_T^{2T} \bigg(\mathfrak{R}\sum_{n\leqs X}\frac{\Lambda(n)V_X(\log n)}{n^{1/2+it}\log n}\bigg)^{2k}dt 
=
J_k(T)+O\Big(C^{2k}J_k(T)^{1-\tfrac{1}{2k}}T^{1/2k}\Big)  
\end{equation}
where 
\[
J_k(T)
=
\int_T^{2T}
\bigg(\frac{1}{2\pi}\int_{-1}^{1}\log|\zeta(\tfrac12+i(t+y))|\widehat{V}_X(y)dy
\bigg)^{2k}dt. 
\]
The result for the real part with $2k$ even will then follow from the subsequent proposition. 
\begin{prop}\label{J prop} For $k\in\tfrac12\mb{N}$ and fixed  $0<\theta<1/100$ we have 
\[
J_k(T)
=
c_k T\bigg(\sum_{p\leqs T^{\theta/k}}\frac{V_X(\log p)^2}{p}\bigg)^k+O(T(Ck)^{4k}(\log\log T)^{k-1/2}) 
\]
where 
\[
c_k=
\begin{cases}
\frac{(2k)!}{2^{2k}k!} & 2k \text{ is even, }\\
0 & 2k \text{ is odd. }
\end{cases}
\]
\end{prop}
For odd $2k$ we  simply compute the left hand side of \eqref{2kth} as $=J_k(T)+O(C^{2k}|J_{k-1/2}|)$. The result then also follows in this case since $J_{k-1/2}(T)\ll T(ck)^{4k}(\log\log T)^{k-1/2}$ by the above proposition and the fact that $V_X(\log p)\ll 1$.

\begin{proof}[Proof of Proposition \ref{J prop}]  We have 
\[
J_k(T)
=
\frac{1}{(2\pi)^{2k}}\int_{[-1,1]^{2k}}\int_T^{2T}\prod_{j=1}^{2k} \log|\zeta(\tfrac12+i(t+y_j))| dt\prod_{j=1}^{2k}\widehat{V}_X(y_j)dy_j. 
\]
Write each term in the product of the inner integral as 
\begin{equation*}\label{log decomp}
\bigg(\log|\zeta(\tfrac12+i(t+y_j))|-\Re\sum_{p\leqs Y}\frac{1}{p^{1/2+i(t+y_j)}}\bigg)+\Re\sum_{p\leqs Y}\frac{1}{p^{1/2+i(t+y_j)}}
\end{equation*}
with $Y=T^{\theta/k}$ for some $0<\theta<1/100$. After accounting for the small shift $y_j\ll1$, Theorem 5.1 of Tsang's thesis \cite{Tsang thesis} gives
\begin{equation}\label{Tsang diff}
\int_T^{2T}\bigg|\mf{F}\log\zeta(\tfrac12+i(t+y_j))-\mf{F}\sum_{p\leqs Y}\frac{1}{p^{1/2+i(t+y_j)}}\bigg|^{2k}dt\ll T(ck)^{\alpha_\mf{F}k}
\end{equation}
for any $k\in\mathbb{N}$ where $\alpha_\Re=4$ and $\alpha_\Im=2$. By H\"older's inequality this extends to real $k>0$. Thus after expanding the product and applying H\"older's inequality in the form $\int\prod_{j=1}^{2k} f_j\ll \prod_{j=1}^{2k}(\int |f_j|^{2k})^{1/2k}$ we have 
\begin{multline*}
J_k(T)
=
\frac{1}{(2\pi)^{2k}}\int_{[-1,1]^{2k}}\int_T^{2T}\prod_{j=1}^{2k}\Re\sum_{p\leqs Y}\frac{1}{p^{1/2+i(t+y_j)}} dt\prod_{j=1}^{2k}\widehat{V}_X(y_j)dy_j
\\
+O(T(ck)^{4k}(\log\log T)^{k-1/2})
\end{multline*}
since $\int_{-1}^1|\widehat{V}_X(y)|dy\ll 1$ and 
\begin{equation}\label{p sum bound}
\int_T^{2T}\bigg|\sum_{p\leqs Y}\frac{1}{p^{1/2+i(t+y)}}\bigg|^{2k}dt\ll T(ck\log\log T)^k
\end{equation}
for all $y\in\mathbb{R}$ and real $k>0$; the integer case following by Lemma \ref{moments lem} below, for example, and  the non-integer case by H\"older's inequality again.

The main term is given by 
\[
\sum_{p_1,\cdots, p_{2k}\leqs Y}\frac{1}{(p_1\cdots p_{2k})^{1/2}}
\frac{1}{(2\pi)^{2k}}\int_{[-1,1]^{2k}}\int_T^{2T}\prod_{j=1}^{2k} \cos((t+y_j)\log p_j)dt\prod_{j=1}^{2k}\widehat{V}_X(y_j)dy_j.
\]
Applying the identity 
\[
\cos((t+y_j)\log p_j)=\cos(t\log p_j)\cos(y_j\log p_j)-\sin(t\log p_j)\sin(y_j\log p_j)
\]
 for each factor and then expanding the product into a sum, we see that any term involving a factor of $\sin(y_j \log p_j)$ vanishes since $\widehat{V}_X(y)$ is an even function. Thus the above is 
\[
\sum_{p_1,\cdots, p_{2k}\leqs Y}\frac{1}{(p_1\cdots p_{2k})^{1/2}}
\frac{1}{(2\pi)^{2k}}\int_{[-1,1]^{2k}}\int_T^{2T}\prod_{j=1}^{2k} \cos(t\log p_j)dt\prod_{j=1}^{2k}\cos(y_j\log p_j)\widehat{V}_X(y_j)dy_j.
\]

Now, by Lemma 4 of \cite{R dev} we have
\[
\int_T^{2T}\prod_{j=1}^{2k} \cos(t\log p_j)dt
=
T f(p_1\cdots p_{2k})+O((2Y)^{2k})
\]
where $f$ is the multiplicative function defined by $f(p^\alpha)=\frac{1}{2^\alpha}\binom{\alpha}{\alpha/2}$ with the convention that the binomial coefficient is zero if $\alpha/2$ is not an integer. After summing over the $p_j$, the error term here contributes $\ll (cY)^{3k}\ll c^{3k}T^{3\theta}$ which is negligible. From the properties of $\widehat{V}_X$ we have 
\begin{align*}
\frac{1}{2\pi}\int_{-1}^1 \cos(y\log p)\widehat{V}_X(y)dy
= &
\frac{1}{2\pi }\int_{-\infty}^\infty \cos(y\log p)\widehat{V}_X(y)dy+O(1/\log X)
\\
= &
 V_X(\log p)+O(1/\log X)
\end{align*}
by Fourier inversion since $V_X$ is real. Aside from a negligible error term, our main term is then
\[
T\sum_{p_1,\cdots, p_{2k}\leqs Y}\frac{f(p_1\cdots p_{2k})\prod_{j=1}^{2k}V_X(\log p)}{(p_1\cdots p_{2k})^{1/2}}.
\]

Since $f(n)$ is supported on squares, if $2k$ is odd there is no main term so we may assume $2k$ is even. In this case, any pairing which results in a fourth power of a prime or higher gives a lower order term since  $\sum_{p\leqs Y}V_X(\log p)^4/p^{4\sigma}\ll1$ and note that there are at most $(2k)!$ such pairings. The main term arises from pairing the $2k$ primes into prime squares, which there are $(2k-1)!!=(2k!)/2^kk!$ ways to do, and since $f(p^2)=1/2$ we see that the above is 
\[
\frac{(2k)!}{2^{2k}k!}\bigg(\sum_{p\leqs Y}\frac{V_X(\log p)^2}{p}\bigg)^k+O((2k)^{2k}(\log\log T)^{k-1}). 
\]
\end{proof}  

\subsection{The absolute $2k$th moment} The case of the absolute moments follows similar lines. By  Proposition \ref{key prop} and H\"older's inequality we have 
\begin{equation}\label{2kth abs}
\int_T^{2T} \bigg|\sum_{n\leqs X}\frac{\Lambda(n)V_X(\log n)}{n^{1/2+it}\log n}\bigg|^{2k}dt 
=
\mc{J}_k(T)+O\Big(C^{2k}\mc{J}_k(T)^{1-\tfrac{1}{2k}}T^{1/2k}\Big)  
\end{equation}
where 
\[
\mc{J}_k(T)
=
\int_T^{2T}
\bigg|\frac{1}{2\pi}\int_{-1}^{1}\log\zeta(\tfrac12+i(t+y))\widehat{V}_X(y)dy
\bigg|^{2k}dt. 
\]
The result then follows from the following proposition. 

\begin{prop}For $k\in\mathbb{N}$ and fixed  $0<\theta<1/100$  we have 
\[
\mc{J}_k(T)
=
k! T\bigg(\sum_{p\leqs T^{\theta/k}}\frac{V_X(\log p)^2}{p}\bigg)^k+O(T(ck)^{4k}(\log\log T)^{k-1/2}). 
\]
\end{prop}
\begin{proof}
We shall only give a sketch since this follows similarly to Proposition \ref{J prop}. We first expand the $2k$th power and push the integral over $t$ through. We then write  
\[
\log\zeta(\tfrac12+i(t+y_j))=\sum_{p\leqs Y}\frac{1}{p^{1/2+i(t+y_j)}}+\Re D_j(t)+\Im D_j(t).  
\] 
where $D_j(t)$ is the difference between the logarithm  and the prime sum. Applying this for each term in the integrand and expanding the product into a sum gives $3^{2k}-1$ error terms involving either the real or imaginary part of $D_j(t)$. These can be dealt as before by using H\"older's inequality, \eqref{Tsang diff} and \eqref{p sum bound} to give an error $\ll T(ck)^{4k}(\log\log T)^{k-1/2}.$
In this way we find 
\begin{multline*}
\mc{J}_k(T)
=
\frac{1}{(2\pi)^{2k}}\int_{[-1,1]^{2k}}\int_T^{2T}\prod_{j=1}^{k}\sum_{p\leqs Y}\frac{1}{p^{1/2+i(t+y_j)}}\prod_{j={k+1}}^{2k}\sum_{p\leqs Y}\frac{1}{p^{1/2-i(t+y_j)}}
 dt\prod_{j=1}^{2k}\widehat{V}_X(y_j)dy_j
\\
+O(T(ck)^{4k}(\log\log T)^{k-1/2}).
\end{multline*}
By distinguishing diagonal and off-diagonal terms the inner integral is
\[
T\sum_{\substack{p_j\leqs Y\\ p_1\cdots p_k=p_{k+1}\cdots p_{2k}}}
\frac{\prod_{j=1}^k p_j^{-iy_j}p_{j+k}^{iy_{j+k}}}{(p_1\cdots p_{2k})^{1/2}}
+
O(Y^{3k})
\] 
since $\sum_{p\leqs Y}p^{-1/2}\ll Y^{1/2}$ and $\int_T^{2T}(p_1\cdots p_k/p_{k+1}\cdots p_{2k})^{-it}dt \ll Y^{2k}$ if $p_1\cdots p_k\neq p_{k+1}\cdots p_{2k}$.
Extending the integrals over $y_j$ to $\pm \infty$ introduces a negligible error and then computing these integrals as before we find that 
\begin{align*}
\mc{J}_k(T)
= &
T\sum_{\substack{p_j\leqs Y\\ p_1\cdots p_k=p_{k+1}\cdots p_{2k}}}
\frac{\prod_{j=1}^k V_X(\log p_j)}{(p_1\cdots p_{2k})^{1/2}}+O(T(ck)^{4k}(\log\log T)^{k-1/2})
\\
= &
Tk!\bigg(\sum_{p\leqs Y}\frac{V_X(\log p)^2}{p}\bigg)^k+O(T(ck)^{4k}(\log\log T)^{k-1/2})
\end{align*}
and so the result follows.
\end{proof}


\section{Large deviations: proof of Theorem \ref{large dev thm}}

Before proving Theorem \ref{log der thm} we give the large deviation estimates since the arguments are relatively short. Throughout we make use of the following standard lemma. 

\begin{lem}\label{moments lem} Let $T$ be large and $2\leqs x\leqs T$. Then for $k$ such that $x^k\leqs T/\log T$ and any complex numbers $a(p)$ we have 
\[
\int_T^{2T}\bigg|\sum_{p\leqs x}\frac{a(p)}{p^{1/2+it}}\bigg|^{2k}dt\ll Tk!\bigg(\sum_{p\leqs x}\frac{|a(p)|^2}{p}\bigg)^k.
\]
\end{lem}
\begin{proof}
See Lemma 3 of \cite{Sound} for example.
\end{proof}

\subsection{The real part} Let us consider the real part first and to simplify things we normalise so that $V(0)=1$. Taking real parts in  Proposition \ref{key prop} we have 
\[
\Re\sum_{n\geqs 2}\frac{\Lambda(n)V_X(\log n)}{n^{1/2+it}\log n}
=
\frac{1}{2\pi}\int_{-1}^1\log|\zeta(\tfrac12+i(y+t))|\widehat{V}_X(y)dy
+
O(1).
\]
since $V_X(x)$ is even and hence $\widehat{V}_X(y)$ is real. If $\wh{V}_X(y)$ is positive then from the bound $\log|\zeta(\tfrac12+it)|\leqs (\tfrac{\log 2}{2}+o(1))\tfrac{\log t}{\log\log t}$ of \cite{CS}, the right hand side is 
\[
\leqs (1+o(1))\frac{\log 2}{2}\frac{\log t}{\log\log t}
\]
for large $t$ by Fourier inversion since $V(0)=1$. Thus, throughout we may assume that 
\[
10\sqrt{\log\log T}\leqs W \leqs \frac{\log 2}{2}\frac{\log T}{\log\log T}.
\] 

We now apply the method of Soundararajan \cite{Sound} to find the frequency of large values of $\log|\zeta(\tfrac12+i(y+t))|$ and hence of our Dirichlet polynomial.
Define the parameter 
\begin{equation}
\label{A}
A
=
\begin{cases}
\tfrac12 \log_3 T &\text{ if } W\leqs \log\log T
\\
\frac{\log\log T}{2W}\log_3 T &\text{ if } \log\log T\leqs W\leqs \tfrac12 \log\log T\log_3 T
\\
1 &\text{ if } W\geqs \tfrac12 \log\log T\log_3 T.
\end{cases}
\end{equation}
and set $Y=T^{A/W}$ and $Z=Y^{1/\log\log T}$. Then from Section 3 of \cite{Sound} we have  for large $t\in [T,2T]$,  
\[
\log|\zeta(\tfrac12+i(y+t))|\leqs f_1(t+y)+f_2(t+y)+(\tfrac{3}{4}-\epsilon)\tfrac{W}{A}+O(\log_3 T) 
\]
 where 
\[
f_1(t)=\bigg|\sum_{p\leqs Z}\frac{1}{p^{1/2+1/2\log Y+it}}\frac{\log(Y/p)}{\log Y}\bigg|
\]
and 
\[
f_2(t)=\bigg|\sum_{Z< p\leqs Y}\frac{1}{p^{1/2+1/2\log Y+it}}\frac{\log(Y/p)}{\log Y}\bigg|.
\]
Consequently, by positivity of $\wh{V}_X(y)$, 
\[
\Re\sum_{n\geqs 2}\frac{\Lambda(n)V_X(\log n)}{n^{1/2+it}\log n}
\leqs
g_1(t)+g_2(t)+(\tfrac{3}{4}-\epsilon)\tfrac{W}{A}+O(\log_3 T) 
\]
for large $t\in [T,2T]$ where 
\[
g_j(t)=\frac{1}{2\pi }\int_{-1}^1f_j(t+y)\widehat{V}_X(y)dy.
\]

Thus, if our sum is greater than $W$ we must have either 
\[
g_1(t)\geqs W\Big(1-\frac{7}{8A}\Big)=:W_1 \qquad \text{ or } \qquad  g_2(t)\geqs \frac{W}{8A}  
\]
We determine the frequency of these conditions by using Chebyshev's inequality along with high moment bounds. The only difference with \cite{Sound} is that we have to deal with the extra integral over $y$. This is easily dealt with however using H\"older's inequality in the form
\begin{align*}
\int_T^{2T}g_j(t)^{2k}dt
= &
\int_T^{2T} \bigg(\frac{1}{2\pi }\int_{-1}^1 f_j(t+y)\wh{V}_X(y)\bigg)^{2k}dt
\\
\leqs & 
 \frac{1}{2\pi }\int_{-1}^1 \Big(\int_T^{2T} f_j(t+y)^{2k}dt\Big) \wh{V}_X(y)dy
\end{align*}
using the fact that $V(0)=1$. By Lemma \ref{moments lem} we find that for all $y\in [-1,1]$,   
\[
\int_T^{2T} f_1(t+y)^{2k_1}dt\ll T k_1!\bigg(\sum_{p\leqs Z}\frac{1}{p}\bigg)^{k_1}\ll T\sqrt{k_1}\bigg(\frac{k_1\log\log T}{e}\bigg)^{k_1}
\]
provided $k_1\leqs \log (T/\log T)/\log Z$, and likewise
 \[
\int_T^{2T} f_2(t+y)^{2k_2}dt\ll T k_2!\bigg(\sum_{Z<p\leqs Y}\frac{1}{p}\bigg)^{k_2}\ll T(k_2\log_3T)^{k_2}
\] 
provided $k_2\leqs \log (T/\log T)/\log Y$.

Thus, choosing $k_2=W/A-1$ we see that the measure of the set of $t\in [T,2T]$ for which $g_2(t)\geqs W/8A$ is
\[
\ll T(W/8A)^{-2k_2}(k_2\log_3T)^{k_2}\ll T\exp(-\tfrac{W}{2A}\log W).
\]
Choosing $k_1=\lfloor W_1^2/\log\log T\rfloor$ if $W<(\log\log T)^2$ and $k_1=\lfloor 10W\rfloor$ when $W>(\log\log T)^2$ we find that the measure of the set of $t\in [T,2T]$ for which $g_1(t)\geqs W_1$ is
\[
\ll
 T\frac{W}{\sqrt{\log\log T}}\exp\bigg(-\frac{W_1^2}{\log\log T}\bigg)
+
T\exp(-4W\log W)
\]
after a short computation. This completes the proof regarding large positive values of the real part.

\subsection{The imaginary part} The imaginary part follows similar lines with some minor modifications which we now describe. We utilise the following formula of Selberg (Theorem 1, \cite{Sel S(t)}) valid for $t\geqs 2$ and $2\leqs x\leqs t^2$:
\[
S(t)
=
\frac{1}{\pi}\Im\sum_{n\leqs x^2}\frac{\Lambda_x(n)}{n^{1/2+1/\log x+it}\log n}
+
O\bigg(\frac{1}{\log x}\bigg|\sum_{n\leqs x^2}\frac{\Lambda_x(n)}{n^{1/2+1/\log x+it}}\bigg|\bigg)
+
O\bigg(\frac{\log t}{\log x}\bigg).
\]
where $\Lambda_x(n)=\Lambda(n)$ for $n\leqs x$ and equals $\Lambda(n)\log(x^2/n)/\log x$ for $x<n\leqs x^2$. From this we find that for $t\in[T/2,4T]$,  
\[
|S(t)| \leqs h_1(t)+h_2(t)+h_3(t)+h_4(t)+C_0\tfrac{\log T}{\log x}
\]
where 
\begin{align*}
h_1(t) 
= &
\frac{1}{\pi}\bigg|\sum_{p\leqs z}\frac{\Lambda_x(p)}{p^{1/2+1/\log x+it}\log p}\bigg|,
\qquad
h_2(t)
=
\frac{1}{\pi}\bigg|\sum_{z<p\leqs x^2}\frac{\Lambda_x(p)}{p^{1/2+1/\log x+it}\log p}\bigg|,
\end{align*}
with $z=x^{2/\log_2 T}$, and 
\begin{align*}
h_3(t) 
= &
\frac{1}{2\pi}\bigg|\sum_{p\leqs x}\frac{\Lambda_x(p^2)}{p^{1+2/\log x+2it}\log p}\bigg|,
\qquad
h_4(t)
=
\frac{B}{\log x}\bigg|\sum_{p\leqs x^2}\frac{\Lambda_x(p)}{p^{1/2+1/\log x+it}}\bigg|
\end{align*}
for some positive constants $B,C_0$. Then taking imaginary parts in Proposition \ref{key prop} and applying the trivial bound we find 
\[
\bigg|\Im\sum_{n\geqs 2}\frac{\Lambda(n)V_X(\log n)}{n^{1/2+it}\log n}\bigg|\leqs \sum_{j=1}^4 \ell_j(t) + C_0\tfrac{\log T}{\log x}
\]
where 
\[
\ell_j(t)=\frac{1}{2}\int_{-1}^1 h_j(t+y)\wh{V}_X(y)dy 
\]
again by positivity of $\wh{V}_X(y)$ and the normalisation $V(0)=1$.

Similarly to before we set $x^2=T^{A^\prime/W}$ where now
\begin{equation}
\label{A prime}
A^\prime
=
\begin{cases}
\tfrac12 \log_3 T &\text{ if } W\leqs \log\log T
\\
\frac{\log\log T}{2W}\log_3 T &\text{ if } \log\log T\leqs W\leqs \tfrac{1}{2(5+2C_0)} \log\log T\log_3 T
\\
5+2C_0 &\text{ if } W\geqs \tfrac{1}{2(5+2C_0)} \log\log T\log_3 T
\end{cases}
\end{equation}
and ask when either 
\[
\ell_1(t)\geqs W\bigg(1-\frac{4+2C_0}{A^{\prime}}\bigg) \qquad \text{ or }\qquad \ell_j(t)\geqs \frac{W}{A^\prime}
\]
for $j=2,3$ or $4$. We may now proceed as before to give the theorem in the case of the imaginary part. Here we note that the terms $\ell_3(t)$, $\ell_4(t)$ can be dealt with using Lemma \ref{moments lem}
along with the bounds 
\[
\sum_{p\leqs x}\frac{1}{p^2}\ll 1,\qquad \sum_{p\leqs x^2}\frac{\log^2 p}{p}\ll \log^2 x
\]
and a similar choice of parameters to $g_2(t)$ of the previous section.


\section{Correlations of $S(t)$} In order to prove Theorem \ref{log der thm} we require precise correlations of $S(t)$. As mentioned in the introduction, for these purposes we adapt a result of Goldston \cite{Goldston} which gives lower order terms in the mean square of $S(t)$. Our result is as follows. 

\begin{prop}\label{S corr prop}Assume RH. Then uniformly in $y_j$ satisfying  $|y_j|\leqs \log T$, we have 
\begin{equation}\begin{split}\label{S corr int 2}
\int_0^TS(t+y_1)S(t+y_2)dt 
= &
\frac{T}{2\pi^2}\int_{{\log 2}}^{\log T} \frac{\cos(u(y_1-y_2))}{u}du
\\
&+\frac{T}{2\pi^2}
\int_1^\infty \frac{F(u)\cos(u(y_1-y_2)\log T)}{u^2}du
+
\frac{T}{2\pi^2}c(y_1,y_2)
\\
&
+O\big(T \tfrac{(y_1-y_2)^2}{1+(y_1-y_2)^2}\big)+O(T|y_1-y_2|\log_2 T)+o(T)
\end{split}
\end{equation}
where $F(u)$ is given by \eqref{F} and 
\begin{multline}\label{c}
c(y_1,y_2)
=
\cos((y_1-y_2)\log 2)\Big(\log\log 2 + \gamma+\sum_{m=2}^\infty\sum_p \frac{1}{mp^m}\Big)
\\
+
\sum_{p}\sum_{m=2}^\infty\frac{\cos(m(y_2-y_1)\log p)}{m^2p^m}.
\end{multline}
\end{prop}

We remark that the second error term in \eqref{S corr int 2} can be made explicit and is given by the integral in formula \eqref{I2} below. Since our arguments closely follow those of Goldston, at certain times we may only describe the main points, referring to \cite{Goldston} for details.  

As usual, we start with an approximate formula. Assuming RH, Lemma 1 and (2.12) of \cite{Goldston} give for $t\geqs 1$, $t\neq \gamma$ and $x\geqs 4$, 
\begin{equation}\label{approx form}
S(t)
=
P_x(t)+Z_x(t)
+
O(1/t(\log x)^2)
+
O(x^{1/2}/(t\log x)^2)
\end{equation}
where 
\begin{equation}
P_x(t)=-\frac{1}{\pi}\sum_{n\leqs x}\frac{\Lambda(n)\sin(t\log n)}{n^{1/2}\log n}f\big(\tfrac{\log n}{\log x}\big),\qquad 
f(u)
=
\frac{\pi }{2}u\cot\big(\tfrac{\pi u}{2}\big)
\end{equation}
and 
\begin{equation}
\,\,\,\,\,\,Z_x(t)=\frac{1}{\pi}\sum_\gamma h((t-\gamma)\log x),\qquad
h(v)=\sin(v)\int_0^\infty \frac{u}{u^2+v^2}\frac{du}{\sinh u}.
\end{equation}
Throughout we shall assume 
\[
\qquad\qquad x=T^\beta,
\]
for some fixed $0<\beta<1/2$, say. Noting that $S(t+y_j)$ is bounded by $\ll\log T$ for $|t|\leqs 1$ we may restrict to the integral over $[1,T]$. As usual, to avoid mixing the sums over primes and zeros, one considers  
\[
L(T):=\int_1^T \Big(S(t+y_1)-P_x(t+y_1)\Big)\Big(S(t+y_2)-P_x(t+y_2)\Big)dt 
\]
which by the above approximate formula is given by 
\[
R(T):=\int_1^T Z_x(t+y_1)Z_x(t+y_2)dt
\]
plus a negligible error which can be dealt with using the Cauchy--Schwarz inequality once we have a bound for $R(T)$.

\subsection{Integral of the sum over zeros} 
In this subsection we consider the integral of the sums over zeros and prove the following.
\begin{lem}\label{R lem}Assume RH. Let $F(u)$ be given by \eqref{F} and let 
\begin{equation}\label{k}
k(u)=
\begin{cases}
\Big(\tfrac{1}{2u}-\tfrac{\pi^2}{2}\cot(\pi^2 u)\Big)^2 \,\, &\text{ if } |u|\leqs 1/2\pi
\\
\tfrac{1}{4u^2} &\,\, \text{ if } |u|>1/2\pi.
\end{cases}
\end{equation}
 Then for $|y_j|\leqs \log T$ and $x=T^\beta$ with fixed $0<\beta<1/2$ we have
\begin{multline*}
R(T)
=
\frac{2T}{(2\pi^2\beta)^2}
\int_{0}^\infty k(u/2\pi \beta) F(u)\cos(u(y_1-y_2)\log T)du
\\
+O\big(T \tfrac{(y_1-y_2)^2}{1+(y_1-y_2)^2}\big)
+O((\log T)^4)
+O(T/\log T)
\end{multline*}
uniformly in $y_j$.
\end{lem}

\begin{proof}
Following a similar argument to section 3 of \cite{Goldston} we split the range of summation to $\gamma\in[0,T]$, $\gamma\in[T,T+3\log T]$, $\gamma\geqs T+3\log T$, $\gamma\in [-3\log T, 0]$, $\gamma\leqs -3\log T$ and use the bounds $h(v)\ll 1/(1+v^2)$, $|y_j|\leqs \log T$ to deduce that 
\[
R(T)
=
\sum_{\gamma_1,\gamma_2\in [0,T]}\int_1^Th((t+y_1-\gamma_1)\log x)h((t+y_2-\gamma_2)\log x)dt+O((\log T)^4).
\] 
With a similar argument, we then extend the range of integration to $t\in(-\infty,\infty)$ at the cost of the same error. Writing the integral as a convolution and following the remainder of section 3 of \cite{Goldston} we find that\footnote{Here, the presence of $-2\pi$ comes from Goldston's use of an alternative convention for the Fourier transform} 
\[
R(T)
=
\frac{1}{\pi^2\log x}\sum_{\gamma_1,\gamma_2\in [0,T]}\widehat{k}(-2\pi(\gamma_1-\gamma_2+y_2-y_1)\log x)+O((\log T)^4).
\] 
with $k(u)$ as in \eqref{k}. 

Next, we wish to add in a factor of $w(\gamma_1-\gamma_2)$ to our sum.
The difference between these two sums is 
\begin{align*}
\sum_{\gamma_1,\gamma_2\in [0,T]}\widehat{k}(-2\pi(\gamma_1-\gamma_2+y_2-y_1)\log x)\cdot\frac{(\gamma_1-\gamma_2)^2}{4+(\gamma_1-\gamma_2)^2}
\end{align*}
and we seek to upper bound this. Write $Y=y_1-y_2$. From formula (4.3) of \cite{Goldston} we have $\widehat{k}(u)\ll 1/(1+u^2)$ and hence the above is 
\begin{multline}\label{zero sum}
\ll\frac{1}{\log^2 x}\sum_{\gamma_1,\gamma_2\in [0,T]}\frac{1}{1+(\gamma_1-\gamma_2-Y)^2\log^2 x}\cdot\frac{(\gamma_1-\gamma_2-Y+Y)^2\log^2 x}{4+(\gamma_1-\gamma_2)^2}
\\
\ll 
\frac{1}{\log^2 x}\sum_{\gamma_1,\gamma_2\in [0,T]}\frac{1}{4+(\gamma_1-\gamma_2)^2}
+
Y^2\sum_{\gamma_1,\gamma_2\in [0,T]}\frac{1}{1+(\gamma_1-\gamma_2-Y)^2\log^2 x}\cdot\frac{1}{4+(\gamma_1-\gamma_2)^2}
\end{multline}
Note we may restrict the second sum to $|(\gamma_1-\gamma_2-Y)\log x|\leqs \tfrac{1}{2}Y\log x$ since otherwise we get something of the same order as the first term. By a short calculation the first term here is
\[
\frac{1}{\log^2 x}\sum_{\gamma_1,\gamma_2\in [0,T]}\frac{1}{4+(\gamma_1-\gamma_2)^2}
\ll 
T\frac{\log^2 T}{\log^2 x}
\ll T. 
\]

We bound the second term by 
\begin{align*}
\ll 
\frac{Y^2}{1+Y^2}\bigg(
\sum_{\substack{0<\gamma_1,\gamma_2\leqs T\\|\gamma_1-\gamma_2-Y|\leqs1/\log x}}1
+
\sum_{n\leqs \tfrac12 Y\log x}\frac{1}{n^2}\sum_{\substack{0<\gamma_1,\gamma_2\leqs T\\n/\log x< |\gamma_1-\gamma_2-Y|\leqs (n+1)/\log x}}1\bigg)
\end{align*}
and note that we may restrict the sum over $\gamma_2$, say, to $\gamma_2> 3\log T$ since
\[
\sum_{\gamma_2\leqs 3\log T} \sum_{\substack{0<\gamma_1\leqs T\\n/\log x< |\gamma_1-\gamma_2-Y|\leqs (n+1)/\log x}}1\ll (\log T)^2\log_2 T
\]
since the number of zeros in a window of length $1/\log x\leqs 1$ around $\gamma_2+Y\pm n/\log x$ is $\ll \log(\gamma_2+2|Y|)\ll\log T$. Then for $0\leqs n\leqs Y/2\log x$ and $\gamma_2\geqs 3|Y|/2$, since 
\begin{align*}
&
N(\gamma_2+Y\pm \tfrac{n}{\log x}+\tfrac{1}{\log x})-N(\gamma_2+Y\pm \tfrac{n}{\log x}-\tfrac{1}{\log x})
\\
= & S(\gamma_2+Y\pm \tfrac{n}{\log x}+\tfrac{1}{\log x})-S(\gamma_2+Y\pm \tfrac{n}{\log x}-\tfrac{1}{\log x})
+
O\Big(\frac{\log(\gamma_2+2|Y|)}{\log x}\Big)
+
O\Big(\frac{1}{\gamma_2}\Big),
\end{align*}
we have 
\begin{align*}
&
\sum_{3\log T\leqs \gamma_2\leqs T} 
\sum_{\substack{0<\gamma_1\leqs T\\n/\log x< |\gamma_1-\gamma_2-Y|\leqs (n+1)/\log x}}1
\ll  
\sum_{3\log T\leqs \gamma_2\leqs T} 
\sum_{\substack{\gamma_2+Y\pm n/\log x-1/\log x\\<\gamma_1\leqs \gamma_2+Y\pm n/\log x+1/\log x}}1
\\
\ll &
\,\,T\frac{\log^2 T}{\log x}
+\bigg|\sum_{3\log T\leqs \gamma_2\leqs T}S(\gamma_2+Y\pm \tfrac{n}{\log x}+\tfrac{1}{\log x})\bigg| 
+\bigg|\sum_{3\log T\leqs \gamma_2\leqs T}S(\gamma_2+Y\pm \tfrac{n}{\log x}-\tfrac{1}{\log x})\bigg|. 
\end{align*}
From the Lemma of \cite{Fujii} we have that for any $a\ll T^A$, with $A$ some positive constant,
\[
\sum_{0<\gamma\leqs T,\,\,\gamma+a>0}S(\gamma+a)\ll T\log T
\] 
and hence putting things together the second term on the right of \eqref{zero sum} is 
\[
\ll \frac{Y^2}{1+Y^2}\cdot T\log T.
\]

Thus we have shown that 
\begin{multline*}
R(T)
=
\frac{1}{\pi^2\log x}\sum_{\gamma_1,\gamma_2\in [0,T]}\widehat{k}(-2\pi (\gamma_1-\gamma_2+y_2-y_1)\log x)\cdot w(\gamma_1-\gamma_2)
\\
+O\big(T \tfrac{(y_1-y_2)^2}{1+(y_1-y_2)^2}\big)
+O((\log T)^4)
+O(T/\log T).
\end{multline*}
Unfolding the integral in $\widehat{k}$ and denoting $e(v)=e^{2\pi i v}$ we have 
\begin{align*}
\frac{1}{\pi^2\log x}&\sum_{\gamma_1,\gamma_2\in [0,T]}\widehat{k}(-2\pi(\gamma_1-\gamma_2+y_2-y_1)\log x)\cdot w(\gamma_1-\gamma_2)
\\
= 
&\frac{1}{\pi^2\log x}
\int_{-\infty}^\infty k(u) \sum_{0<\gamma_1,\gamma_2\leqs T}e(-u(\gamma_1-\gamma_2+y_2-y_1)\log x)w(\gamma_1-\gamma_2)du
\\
= &
\frac{T\log T}{2\pi^3\log x}
\int_{-\infty}^\infty k(u) F(2\pi \beta u)e(u(y_1-y_2)\log x)du
\\
= &
\frac{2T}{(2\pi^2\beta)^2}
\int_{0}^\infty k(u/2\pi \beta) F(u)\cos(u(y_1-y_2)\log T)du
\end{align*}
since $k$ and $F$ are even.
\end{proof}

\subsection{Computing $L(T)$}
We now turn our attention to $L(T)$. Expanding we find that 
\[
L(T)
=
\int_1^TS(t+y_1)S(t+y_2)dt+G(T,y_1,y_2)+H(T,y_1,y_2)+H(T,y_2,y_1)
\]
where 
\begin{multline*}
\qquad\qquad G(T,y_1,y_2)
=
\frac{1}{\pi^2} \sum_{n_1,n_2\leqs x}\frac{\Lambda(n_1)\Lambda(n_2)}{(n_1n_2)^{1/2}\log n_1\log n_2}f\big(\tfrac{\log n_1}{\log x}\big)f\big(\tfrac{\log n_2}{\log x}\big)
\\
\times \int_1^T\sin((t+y_1)\log n_1)\sin((t+y_2)\log n_2)dt.
\end{multline*}
and
\[
H(T,y_1,y_2)
=
\int_1^TS(t+y_1)\cdot \frac{1}{\pi}\sum_{n\leqs x}\frac{\Lambda(n)\sin((t+y_2)\log n)}{n^{1/2}\log n}f\big(\tfrac{\log n}{\log x}\big)dt
\]
Thus it remains to compute $G$ and $H$.

\begin{lem}\label{GH lem}
Assume RH and that $|y_j|\leqs \log T$. Then 
\[
G(T,y_1,y_2)
=
\frac{T}{2\pi^2}\sum_{n\leqs x}\frac{\Lambda(n)^2\cos((y_2-y_1)\log n)}{n\log^2n}f\big(\tfrac{\log n}{\log x}\big)^2
+
O(x^{2+\epsilon})
\]
and
\[
H(T,y_1,y_2)
=
-\frac{T}{2\pi^2}\sum_{n\leqs x}\frac{\Lambda(n)^2\cos((y_2-y_1)\log n)}{n\log^2 n}f\big(\tfrac{\log n}{\log x}\big)+O(x^{2+\epsilon}).
\]
\end{lem}
\begin{proof}
The proof for $G$ follows Lemma 6 of \cite{Goldston} and proceeds in the familiar way. The diagonal terms $n_1=n_2$ give the main contribution of 
\[
\frac{1}{2\pi^2}(T+O(1))\sum_{n\leqs x}\frac{\Lambda(n)^2\cos((y_2-y_1)\log n)}{n\log^2n}f\big(\tfrac{\log n}{\log x}\big)^2
\]
since 
\[
\int_1^T\sin((t+y_1)\log n)\sin((t+y_2)\log n)dt =\frac{1}{2}T\cos((y_2-y_1)\log n)+O(1).
\]
The off-diagonals can be estimated as in \cite{Goldston} and are $\ll x^2$.

For $H$ we swap the order of summation and integration to find
\begin{align*}
H(T,y_1,y_2)
= &
\frac{1}{\pi}\sum_{n\leqs x}\frac{\Lambda(n)}{n^{1/2}\log n}f\big(\tfrac{\log n}{\log x}\big)\cdot \int_1^TS(t+y_1)\sin((t+y_2)\log n)dt
\\
= &
\frac{1}{\pi}\sum_{n\leqs x}\frac{\Lambda(n)}{n^{1/2}\log n}f\big(\tfrac{\log n}{\log x}\big)\cdot \int_1^TS(t)\sin((t+y_2-y_1)\log n)dt+O(x^{1/2}(\log T)^2)
\end{align*}
since $f(\log n/\log x)\ll 1$ for $n\leqs x$, $y_j\ll \log T$ and $S(t)\ll \log T$. The above integral is given by 
\[
\cos((y_2-y_1)\log n)\int_1^TS(t)\sin(t\log n)dt
+
\sin((y_2-y_1)\log n)\int_1^TS(t)\cos(t\log n)dt.
\]

Now,  Lemma $\gamma$ of \cite{Titch} (see also (6.3) of \cite{Goldston}) states that assuming RH,
\[
\int_1^TS(t)\sin(t\log n)dt
=
-\frac{T}{2\pi}\frac{\Lambda(n)}{n^{1/2}\log n}+O(n^{3/2}\log T)
\]
and the same proof gives that 
\[
\int_1^TS(t)\cos(t\log n)dt
=
O(n^{3/2}\log T).
\]
Hence 
\[
H(T,y_1,y_2)
=
-\frac{T}{2\pi^2}\sum_{n\leqs x}\frac{\Lambda(n)^2\cos((y_2-y_1)\log n)}{n\log^2 n}f\big(\tfrac{\log n}{\log x}\big)+O(x^{2+\epsilon})
\]
and the result follows.
\end{proof}

\subsection{Combining formulae} Combining Lemmas \ref{R lem} and \ref{GH lem} with $\beta<1/2$ through the approximate formula \eqref{approx form} we find that 
\begin{align}\label{S corr int}
\int_0^TS(t+y_1)S(t+y_2)dt 
= &
\frac{T}{\pi^2}\sum_{n\leqs T^\beta}\frac{\Lambda(n)^2\cos((y_2-y_1)\log n)}{n\log^2n}\Big[f\big(\tfrac{\log n}{\beta\log T}\big)-\tfrac12f\big(\tfrac{\log n}{\beta\log T}\big)^2\Big]
\nonumber
\\
&
+
\frac{2T}{(2\pi^2\beta)^2}
\int_{0}^\infty k(u/2\pi \beta) F(u)\cos(u(y_1-y_2)\log T)du
\nonumber
\\
&
+O\big(T \tfrac{(y_1-y_2)^2}{1+(y_1-y_2)^2}\big)
+O\Big(T^{1/2+\beta/2}\Big(\int_{0}^\infty k(u/2\pi \beta) F(u)du\Big)^{1/2}\Big)+o(T)
\end{align}
where the second to last error term arises from the error terms of \eqref{approx form}, the Cauchy--Schwarz inequality and our formula for $R(T)$ with $y_1=y_2$.
From \eqref{F int} and the definition of $k(u)$ given in \eqref{k} we immediately see that this error term is $O(T^{3/4})$.

Let us deal with the integral involving $F(u)$ first. 
Writing $C(u)=\cos(u(y_1-y_2)\log T)$ for short, from the asymptotic formula for $F(u)$ given in \eqref{F 1} we find that 
\begin{equation}\label{F int decomp}
\begin{split}
&
\frac{2T}{(2\pi^2\beta)^2}
\int_{0}^\infty k(u/2\pi \beta) F(u)\cos(u(y_1-y_2)\log T)du
\\
= &
\frac{T}{2\pi^2}
\int_0^\beta \frac{1}{u^2}\Big(1-\tfrac{\pi u}{2\beta}\cot(\tfrac{\pi u}{2\beta})\Big)^2\big(u+o(1)+(1+o(1))T^{-2u}\log T\big)C(u)du
\\
& +
\frac{T}{2\pi^2}
\int_\beta^1 \frac{1}{u^2}(u+o(1)+(1+o(1))T^{-2u}\log T)C(u)du
+
\frac{T}{2\pi^2}
\int_1^\infty \frac{F(u)C(u)}{u^2}du
\\
= &
\frac{T}{2\pi^2}
\int_0^\beta \frac{1}{u}\big(1-f(\tfrac{u}{\beta})\big)^2C(u)du
+
\frac{T}{2\pi^2}
\int_\beta^1 \frac{C(u)}{u}du
\\
&
+\frac{T}{2\pi^2}
\int_1^\infty \frac{F(u)C(u)}{u^2}du+o(T)
\end{split}
\end{equation}
recalling that $f(u)=\frac{\pi }{2}u\cot\big(\tfrac{\pi u}{2}\big)$. Note that from the asymptotic formula $f(u)=1+O(u^2)$ the integral involving $T^{-2u}\log T$ was $o(T)$ and also that the first integral on the right may be restricted to $\tfrac{\log 2}{\log T}\leqs u\leqs \beta$ at the cost of an error $o(T)$.

Let us compare this to the sum over primes and prime powers. This is given by 
\begin{equation}\label{prime sum decomp}
\begin{split}
&
\frac{T}{\pi^2}\sum_{n\leqs x}\frac{\Lambda(n)^2\cos((y_2-y_1)\log n)}{n\log^2n}\Big[f\big(\tfrac{\log n}{\log x}\big)-\tfrac12f\big(\tfrac{\log n}{\log x}\big)^2\Big]
\\
= &
\frac{T}{\pi^2}\sum_{p\leqs x}\frac{\cos((y_2-y_1)\log p)}{p}\Big[f\big(\tfrac{\log p}{\log x}\big)-\tfrac12f\big(\tfrac{\log p}{\log x}\big)^2\Big]
\\
& +
\frac{T}{2\pi^2}\sum_{p}\sum_{m=2}^\infty\frac{\cos(m(y_2-y_1)\log p)}{m^2p^m}
+o(T)
\end{split}
\end{equation}
where for the sum over prime powers we have used that $f(u)\ll 1$ for $u\leqs 1$ and $f(m\log p/\log x)\to1$ as $x\to\infty$ for any fixed $p^m$. Following \cite{Goldston} we write 
\begin{equation}\label{T(u)}
T(u)=\sum_{2\leqs p\leqs u}\frac1p=\log\log u + \gamma+\sum_{m=2}^\infty\sum_p \frac{1}{mp^m}+r(u)
\end{equation}
where $r(u)\ll 1/\log u$. Then the sum over primes in \eqref{prime sum decomp} is 
\begin{equation*}\label{prime sum int}
\begin{split}
&
\frac{T}{\pi^2}\int_2^x \cos((y_1-y_2)\log u)\Big[f\big(\tfrac{\log u}{\log x}\big)-\tfrac12f\big(\tfrac{\log u}{\log x}\big)^2\Big]dT(u)
\\
= &
\frac{T}{\pi^2}\int_2^x \cos((y_1-y_2)\log u)
\Big[f\big(\tfrac{\log u}{\log x}\big)-\tfrac12f\big(\tfrac{\log u}{\log x}\big)^2\Big]
\Big(\frac{d(\log u)}{\log u}+dr(u)\Big)
\\
= & 
I_1+I_2,
\end{split}
\end{equation*}
say.

In $I_1$ we substitute $u\mapsto T^u$ to find
\begin{equation}\label{I_1}
I_1
=
\frac{T}{2\pi^2}\int_{\tfrac{\log 2}{\log T}}^\beta C(u)\big(2f\big(\tfrac{u}{\beta}\big)-f\big(\tfrac{u}{\beta}\big)^2\big) \frac{du}{u}
\end{equation}
and note that when this is added to the first term on the right of \eqref{F int decomp} the terms involving $f$ cancel.

In $I_2$ we integrate by parts to see 
\begin{multline*}
I_2
= 
-\frac{T}{\pi^2}\cos((y_1-y_2)\log 2)\Big[f\big(\tfrac{\log 2}{\log x}\big)-\tfrac12f\big(\tfrac{\log 2}{\log x}\big)^2\Big]r(2)+O(T/\log x)
\\
 -
\frac{T}{\pi^2}\int_2^x r(u)\frac{d}{dx} \bigg[\cos((y_1-y_2)\log u)\Big(f\big(\tfrac{\log u}{\log x}\big)-\tfrac12f\big(\tfrac{\log u}{\log x}\big)^2\Big)\bigg]du
\end{multline*}
Applying \eqref{T(u)} to find the value of $r(2)$ and using the fact that $f(\log 2/\log x)=1+O(1/\log x)$ along with the bounds $(d/du)f(\log u/\log x)\ll 1/(u\log x)$ and $r(u)\ll 1/\log u$ 
we see that 
\begin{multline}\label{I2}
I_2
= 
\frac{T}{2\pi^2}\cos((y_1-y_2)\log 2)\Big(\log\log 2 + \gamma+\sum_{m=2}^\infty\sum_p \frac{1}{mp^m}\Big)
\\
 +(y_1-y_2)\frac{T}{\pi^2}\int_2^x r(u)\sin((y_1-y_2)\log u)\Big[f\big(\tfrac{\log u}{\log x}\big)-\tfrac12f\big(\tfrac{\log u}{\log x}\big)^2\Big]\frac{du}{u}+O\big(\tfrac{T\log_2x}{\log x}\big).
\end{multline}
Note the second term here is $\ll |y_1-y_2|T\log\log T$. Combining this with \eqref{F int decomp}, \eqref{prime sum decomp} and \eqref{I_1} in \eqref{S corr int} we acquire
\begin{align}\label{S corr int 2.5}
\int_0^TS(t+y_1)S(t+y_2)dt 
= &
\frac{T}{2\pi^2}\int_{\tfrac{\log 2}{\log T}}^1 \frac{C(u)}{u}du
+\frac{T}{2\pi^2}
\int_1^\infty \frac{F(u)C(u)}{u^2}du
+
\frac{T}{2\pi^2}c(y_1,y_2)
\nonumber
\\
&
+O\big(T \tfrac{(y_1-y_2)^2}{1+(y_1-y_2)^2}\big)+O(T|y_1-y_2|\log_2 T)+o(T)
\end{align}
where $c(y_1,y_2)$ is given by \eqref{c}. Proposition \ref{S corr prop} then follows.


\section{Proof of Theorem \ref{log der thm}}

Recalling Proposition \ref{log der prop} we have 
\[
\sum_{n\leqs X}\frac{\Lambda(n)}{n^{1/2+it}}(1-\tfrac{\log n}{\log X}) 
=
\int_{-\log T/\log_2 T}^{\log T/\log_2 T}S(t+y)f_X(y)dy+O(1)
\]
with
\[
f_X(y)=\frac{X^{iy}(2-iy\log X)-2-iy\log X}{y^3\log X}.
\]
Thus by Cauchy--Schwarz, 
\[
\int_T^{2T}\bigg|\sum_{n\leqs X}\frac{\Lambda(n)}{n^{1/2+it}}(1-\tfrac{\log n}{\log X}) 
\bigg|^2dt
=
\mc{K}(T)+O(T^{1/2}\mc{K}(T)^{1/2})
\]
where 
\begin{align*}
\mc{K}(T)
= &
\int_T^{2T}\bigg|\int_{-\log T/\log_2 T}^{\log T/\log_2 T}S(t+y)f_X(y)dy\bigg|^2dt
\\
= &
\int_{|y_j|\leqs \tfrac{\log T}{\log_2 T}}\bigg(\int_T^{2T}S(t+y_1)S(t+y_2)dt\bigg)f_X(y_1)\ol{f_X(y_2)}dy_1dy_2.
\end{align*}
Applying Proposition \ref{S corr prop} the inner integral is 
\begin{equation}\begin{split}\label{S corr int 3}
&\frac{T}{2\pi^2}\int_{{\log 2}}^{\log T} \frac{\cos(u(y_1-y_2))}{u}du
+\frac{1}{2\pi^2}
\int_1^\infty {tF(u,t)\cos(u(y_1-y_2)\log t)\Big|_{t=T}^{2T}}\frac{du}{u^2}
\\
&
+
\frac{T}{2\pi^2}c(y_1,y_2)
+O\big(T \tfrac{(y_1-y_2)^2}{1+(y_1-y_2)^2}\big)+O(T|y_1-y_2|\log_2 T)+o(T)
\end{split}
\end{equation}
since $\int_{\log T}^{\log 2T}du/u\ll1/\log T$. We first estimate the contribution from the error terms here. 

Substituting $y\mapsto y/\log X$ in each of the following integrals we find 
\[
\int_{-\log T/\log_2 T}^{\log T/\log_2 T}|f_X(y)|dy\ll\log T,\qquad  \int_{-\log T/\log_2 T}^{\log T/\log_2 T}|yf_X(y)|dy\ll \log_2 T
\]
and 
\[
 \int_{-\log T/\log_2 T}^{\log T/\log_2 T}|y^2f_X(y)|dy\ll \log T/\log_2 T.
\]
Therefore, the error terms of \eqref{S corr int 3} contribute $o( T\log^2 T)$ to $\mc{K}(T)$. The main terms all involve a factor of cosine and for these we use the following lemma. 
\begin{lem}
For $x\in\mathbb{R}$ we have 
\[
\int_{-\log T/\log_2 T}^{\log T/\log_2 T} e^{-ixy}f_X(y)dy
=
\mathds{1}_{0<x<\log X} \cdot 2\pi ix\frac{\log(X/e^x)}{\log X} +O\big(\tfrac{\log_2 T}{\log T}\big)
\]
\end{lem}

\begin{proof}
Extending the integrals to $\pm \infty$ introduces an  error $\ll \log_2 T/\log T$ since $f_X(y)\ll 1/y^2$ for $y\gg 1$. Then we apply the formula 
\[\qquad 
f_X(y)=\frac{i}{2\pi i \log X }\int_{(c)}\frac{1}{(s-iy)^2} \frac{X^s}{s^2}ds, \qquad c>0 
\]
which follows by calculations akin to those of Proposition \ref{log der prop}. By absolute convergence we may interchange the order of integration to find that our integral is given by 
\[
\frac{-i}{2\pi i \log X }\int_{(c)}\bigg(\int_{-\infty}^\infty \frac{e^{-ixy}}{(y-s/i)^2}dy\bigg) \frac{X^s}{s^2}ds.
\]
The inner integral may be computed by contour integration in the usual way. Note that we have a double pole at $y=s/i$ which is in the lower half-plane and thus we only get a contribution if $x>0$. In this way we find  the above is 
\[
\frac{\mathds{1}_{x>0} \cdot 2\pi i x}{2\pi i \log X }\int_{(c)}\frac{(X/e^x)^s}{s^2}ds
=
\mathds{1}_{0<x<\log X} \cdot 2\pi ix\frac{\log(X/e^x)}{\log X}. 
\]
\end{proof}

From the lemma we find that for any $x>0$
\begin{align*}
\int_{|y_j|\leqs \log T/\log_2 T} & \cos(x(y_1-y_2))f_X(y_1)\ol{f_X(y_2)}dy_1dy_2
\\
= & 
\frac{1}{2}\bigg|\int_{-\log T/\log_2 T}^{\log T/\log_2 T}e^{ixy}f_X(y)dy\bigg|^2
+
\frac{1}{2}\bigg|\int_{-\log T/\log_2 T}^{\log T/\log_2 T}e^{-ixy}f_X(y)dy\bigg|^2 
\\
= &
\frac{1}{2}\,\mathds{1}_{x<\log X} \cdot (2\pi x)^2\,\frac{\log^2(X/e^x)}{\log^2 X} 
+O\big(\mathds{1}_{x<\log X} \cdot x\tfrac{\log_2 T}{\log T}\big)
+O\Big(\big(\tfrac{\log_2 T}{\log T}\big)^2\Big).
\end{align*}
We apply these results to \eqref{S corr int 3}. First, we find that 
\begin{align*}
&
\frac{T}{2\pi^2}\int_{|y_j|\leqs \log T/\log_2 T}  \int_{{\log 2}}^{\log T} \frac{\cos(u(y_1-y_2))}{u}du
f_X(y_1)\ol{f_X(y_2)}dy_1dy_2
\\
= &
T\int_{{\log 2}}^{\min(\log T,\,\log X)} u\big(1-\tfrac{u}{\log X}\big)^2du
+
O(T\log_2 T).
\end{align*}
After a change of variables and a short calculation with the prime number theorem this is equal to
\[
T\sum_{p\leqs \min(T,\,X)}\frac{\log^2 p}{p}\Big(1-\tfrac{\log p}{\log X}\Big)^2
+
o(T(\log T)^2).
\]
For the integral involving $F(u)$ we have for $t=T$ or $2T$,
\begin{align*}
&
\frac{t}{2\pi^2}\int_{|y_j|\leqs \log T/\log_2 T}\int_1^\infty {F(u,t)\cos(u(y_1-y_2)\log t)}\frac{du}{u^2}f_X(y_1)\ol{f_X(y_2)}dy_1dy_2
\\
= &
\mathds{1}_{X\geqs t}\cdot {t\log^2 t}\int_1^{\tfrac{\log X}{\log t}}F(u,t)\big(1-\tfrac{u\log t}{\log X}\big)^2du+O(T\log_2 T)
\end{align*}
using \eqref{F int} in the error term. Finally, from the definition of $c(y_1,y_2)$ given in \eqref{c} we have 
\begin{multline*}
\frac{T}{2\pi^2} \int_{|y_j|\leqs \log T/\log_2 T} c(y_1,y_2)f_X(y_1)\ol{f_X(y_2)}dy_1dy_2
=
\\
=
{T}\sum_{\substack{p,m\geqs 2\\p^m\leqs X}}\frac{\log^2 p}{p^m}\frac{\log^2(X/p^m)}{\log^2 X} 
+O\big(T\tfrac{\log_2 T}{\log T}\sum_{\substack{p,m\geqs 2\\p^m\leqs X}}\frac{\log p}{mp^m}\big)+O(T)\ll T.
\end{multline*}
Combining the above three formulae gives Theorem \ref{log der thm}.


\end{document}